\documentclass{amsart}

\usepackage{verbatim}
\usepackage{graphicx}
\usepackage{amsfonts,amsmath,latexsym,amssymb,amsthm}
\usepackage{geometry}
\newcounter{lemmacounter}

\newcounter{dummycounter}
\newcounter{propcounter}

\newcounter{emptycounter}

\newcounter{probcounter}

\newtheorem{lemma}[lemmacounter]{Lemma}

\newtheorem{proposition}[propcounter]{Proposition}

\newtheorem{problem}[probcounter]{Problem}

\newcounter{eqncounter}

\numberwithin{equation}{eqncounter}

\def\IR{\mathbb R}

\def\IN{\mathbb N}

\def\IQ{\mathbb Q}

\def\bz{{\bf z}}

\def\bz0{{\bf z^{(0)}}}
\def\bz1{{\bf z^{(1)}}}

\def\grmnm1{\mathcal{M}^*_K(n-1,X)}

\def\grmnmnew1{\mathcal{M}^*_K(n-1,X^n)}

\renewcommand{\vec}[1]{\mbox{\boldmath$#1$}}
\def\Oseen{{\mathcal{O}}}

\def\C{{\mathfrak{C}}}

\def\D{{\mathfrak{D}}}

\let\rho\varrho
\def\ta_0{\tau}

\def\g0{\ta_0\sigma(\C_0^{-1}\D)}
\def\G0{\tau\Lambda(\D)}

\def\Qbar{\overline{\IQ}}

\def\v0{\vec{0}}
\def\d_1{\kappa}

\title{Lehmer without Bogomolov}
\author{Fabien Pazuki, Niclas Technau, and Martin Widmer}
\address{Fabien Pazuki. University of Copenhagen, Institute of Mathematics, Universitetsparken 5, 2100 Copenhagen, Denmark, and Universit\'e de Bordeaux, IMB, 351, cours de la Lib\'eration, 33400 Talence, France.}
\email{fpazuki@math.ku.dk}
\address{Niclas Technau. Department of Mathematics, University of Wisconsin-Madison, 480 Lincoln Dr, Madison, WI-53706, USA}
\email{technau@wisc.edu}
\address{Martin Widmer. Department of Mathematics, Royal Holloway, University of London,
Egham, Surrey, TW20 0EX, United Kingdom}
\email{Martin.Widmer@rhul.ac.uk}

\begin{document}


\thanks{NT is supported by the Austrian Science Fund (FWF): 
project J 4464-N}

\maketitle

\begin{abstract}
We construct fields of algebraic numbers that have the Lehmer property but not the Bogomolov property. This answers a recent implicit question of Pengo and the first author.
\end{abstract}

\section{Introduction}
For a subset $S\subset \Qbar$ of  the algebraic numbers $\Qbar$ and a function $f:\Qbar\to \IR$ we say $S$ is $f$-Northcott  if $\{\alpha\in S; f(\alpha)\leq X\}$ is finite for each $X\in \IR$. And we say
$S$ is $f$-Bogomolov if $0$ is not an accumulation point of $f(S)$ (wrt the usual topology on $\IR$). 
Here we are interested only in the case $f(\alpha)=(\deg \alpha)^\gamma h(\alpha)$ 
for some fixed $\gamma\in \IR$ where $h(\cdot)$ denotes the absolute logarithmic Weil height on $\Qbar$.

We say  $S$ is $\gamma$-Northcott (or  $\gamma$-Bogomolov) if  $S$ is $f$-Northcott  (or $f$-Bogomolov) for  $f(\alpha)=(\deg \alpha)^\gamma h(\alpha)$. 
So $0$-Northcott and  $0$-Bogomolov are the usual Northcott  and  Bogomolov property (formally introduced by Bombieri and Zannier \cite{BoZa}, and studied, e.g.,  in \cite{BoZa,DvZa,AmorosoDavidZannier, AmDv, AmDa,HabeggerQEtor,CheccoliWidmer,Fehm,Pottmeyer2015,Grizzard2015}), whereas  $1$-Bogomolov was recently introduced by Pazuki and Pengo in \cite{PazukiPengo20} as Lehmer property. 
The authors of the latter article implicitly raised the problem of constructing a field with Lehmer  property that fails to have Bogomolov property. 
Of course, Lehmer's conjecture claims that $\Qbar$ itself is such a field but Lehmer's conjecture is still open.
\begin{problem}\label{QNBG}
Construct a subfield of $\Qbar$ with Lehmer property that fails to have Bogomolov property.
\end{problem}

Here we prove the following stronger result.

\begin{proposition}\label{PropNBG}
For each $0<\gamma\leq 1$ and each $\epsilon>0$ one can construct a field $L$ such that $L$ is  $\gamma$-Northcott but not $(\gamma-\epsilon)$-Bogomolov. 
\end{proposition}
Since $\gamma$-Northcott implies $\gamma$-Bogomolov it follows in particular, that one can construct a field with Lehmer but not Bogomolov property.
Here is a more explicit form of the previous proposition.
\begin{proposition}\label{PropNBGexplicit}
Let $0<\gamma\leq 1$, $\epsilon>0$ and choose sequences of primes $(d_i)_i$ and $(p_i)_i$ with $d_i\geq 2d_{i-1}$, and $e^{d_i^{1-\gamma+\epsilon/2}}\leq p_i\leq 2e^{d_i^{1-\gamma+\epsilon/2}}$.
Then $\IQ(p_i^{1/d_i}; i\in \IN)$ is  $\gamma$-Northcott but not $(\gamma-\epsilon)$-Bogomolov. 
\end{proposition}

\section{Proofs}
The following lemma  uses the method from \cite{WidmerPropN}.
\begin{lemma}\label{LemNBG}
Let $0<\gamma\leq 1$, and let $p_1,p_2,p_3,\ldots$ and $d_1,d_2,d_3,\ldots$ be sequences of positive rational primes such that ${d_i}^{\gamma-1}(\log p_i-\log d_i)\rightarrow \infty$ as $i\rightarrow \infty$, and
$p_i\notin \{d_1,p_1,\ldots,d_{i-1},p_{i-1}\}$ for all $i>i_0$.
Then $L=\IQ(p_i^{1/d_i}; i\in \IN)$ is  $\gamma$-Northcott. 
\end{lemma}
\begin{proof}
Set $K_0=\IQ$ and $K_{i}=K_{i-1}(p_{i}^{1/d_{i}})$. Suppose $L$ does not have $\gamma$-Northcott. Then there exists $X>0$ and 
a sequence $\{\alpha_j\}\subset L$ of pairwise distinct elements with $(\deg \alpha_j)^\gamma h(\alpha_j)\leq X$. 
For $\alpha_j$ we set $i=i(\alpha_j):=\min\{l; \alpha_j\in K_l\}$. By Northcott's Theorem we infer that $i\rightarrow \infty$ as $j\rightarrow \infty$.
From now on we assume $i>i_0$.
Next note that only primes in $\{d_1,p_1,\ldots,d_{i-1},p_{i-1}\}$ can ramify in $K_{i-1}$. We conclude that $p_i$ is unramified in $K_{i-1}$, 
and hence $x^{d_i}-p_i$ is an Eisenstein polynomial in $\Oseen_{K_{i-1}}[x]$. Thus, $[K_i:K_{i-1}]=d_i$ is prime, and we conclude that
$K_{i-1}(\alpha_j)=K_i$. An inequality of Silverman \cite[Theorem 2]{Silverman} (see also \cite[(5)]{WidmerPropN}) implies that
\begin{alignat*}1
h(\alpha_j)\geq 
\frac{\log (N_{K_{i-1}/\IQ}(D_{K_i/K_{i-1}}))}{2[K_{i-1}:\IQ]d_i(d_i-1)}-\frac{\log d_i}{2(d_i-1)},
\end{alignat*}
where $N_{K_{i-1}/\IQ}(\cdot)$ denotes the norm and $D_{K_i/K_{i-1}}$ denotes the relative discriminant.
A straightforward calculation shows (see \cite[Proof of Theorem 4]{WidmerPropN}) that $p_{i}^{[K_{i-1}:\IQ](d_i-1)}$ divides $N_{K_{i-1}/\IQ}(D_{K_i/K_{i-1}})$. Hence,
\begin{alignat*}1
h(\alpha_j)\geq \frac{\log p_i}{2d_i}-\frac{\log d_i}{2(d_i-1)}=\frac{1}{2d_i^{\gamma}} \left(\frac{\log p_i-\log d_i}{d_i^{1-\gamma}}-\frac{\log d_i}{d_i^{2-\gamma}(1-1/d_i)}\right)
\geq \frac{1}{2d_i^{\gamma}} \left(\frac{\log p_i-\log d_i}{d_i^{1-\gamma}}-1\right).
\end{alignat*}
Finally, we note that $\deg(\alpha_j)\geq [K_{i-1}(\alpha_j):K_{i-1}]=d_i$ thus
\begin{alignat*}1
X\geq (\deg \alpha_j)^\gamma h(\alpha_j)\geq \frac{1}{2}\left(\frac{\log p_i-\log d_i}{d_i^{1-\gamma}}-1\right)\rightarrow \infty \text{ as } j\rightarrow \infty.
\end{alignat*}
This yields the desired contradiction, and hence proves that  $L$ is  $\gamma$-Northcott
\end{proof}

\begin{lemma}\label{Lem2NBG}
Let $0<\gamma\leq 1$, $\epsilon>0$, and let $p_1,p_2,p_3,\ldots$ and $d_1,d_2,d_3,\ldots$ be sequences of positive rational primes such that ${d_i}^{\gamma-1}(\log p_i-\log d_i)\rightarrow \infty$ 
and ${d_i}^{\gamma-\epsilon -1}\log p_i\rightarrow 0$ as $i\rightarrow \infty$,  and
$p_i\notin \{d_1,p_1,\ldots,d_{i-1},p_{i-1}\}$ for all $i>i_0$.
Then $L=\IQ(p_i^{1/d_i}; i\in \IN)$ is  $\gamma$-Northcott but not $(\gamma-\epsilon)$-Bogomolov. 
\end{lemma}
\begin{proof}
Since $(\deg p_i^{1/d_i})^{\gamma-\epsilon} h(p_i^{1/d_i})={d_i}^{\gamma-\epsilon -1}\log p_i\rightarrow 0$
the claim follows immediately from Lemma \ref{LemNBG}.
\end{proof}
Proposition \ref{PropNBGexplicit} now follows easily from Lemma \ref{Lem2NBG} as  any sequences of primes $d_i$ with $d_i\geq 2d_{i-1}$, and  $p_i$ with $e^{d_i^{1-\gamma+\epsilon/2}}\leq p_i\leq 2e^{d_i^{1-\gamma+\epsilon/2}}$
satisfy the required conditions of Lemma \ref{Lem2NBG}. 

Finally, note that  if $\gamma=1$ we can take any sequence of primes $d_i$ with $d_i\geq 2d_{i-1}$, take $\delta>1$,
and choose primes $p_i$ with $d_i^{\delta}\leq p_i\leq 2d_i^{\delta}$. 

\bibliographystyle{amsplain}
\bibliography{literature}

\end{document}